\documentclass[11pt]{amsart}

\usepackage{dsfont}
\iftrue
\makeatletter
\def\@settitle{%
  \vspace*{-20pt}
  \begin{flushleft}%
    \baselineskip14\p@\relax
    \normalfont\bfseries\LARGE
    \@title
  \end{flushleft}%
}
\def\@setauthors{%
  \begingroup
  \def\thanks{\protect\thanks@warning}%
  \trivlist
  \large \@topsep30\p@\relax
  \advance\@topsep by -\baselineskip
  \item\relax
  \author@andify\authors
  \def\\{\protect\linebreak}%
  \authors
  \ifx\@empty\contribs
  \else
    ,\penalty-3 \space \@setcontribs
    \@closetoccontribs
  \fi
  \normalfont
  \@setaddresses
  \endtrivlist
  \endgroup
}
\def\@setaddresses{\par
  \nobreak \begingroup\raggedright
  \small
  \def\author##1{\nobreak\addvspace\smallskipamount}%
  \def\\{\unskip, \ignorespaces}%
  \interlinepenalty\@M
  \def\address##1##2{\begingroup
    \par\addvspace\bigskipamount\noindent
    \@ifnotempty{##1}{(\ignorespaces##1\unskip) }%
    {\ignorespaces##2}\par\endgroup}%
  \def\curraddr##1##2{\begingroup
    \@ifnotempty{##2}{\nobreak\noindent\curraddrname
      \@ifnotempty{##1}{, \ignorespaces##1\unskip}\/:\space
      ##2\par}\endgroup}%
  \def\email##1##2{\begingroup
    \@ifnotempty{##2}{\smallskip\nobreak\noindent E-mail address%
      \@ifnotempty{##1}{, \ignorespaces##1\unskip}\/:\space
      \ttfamily##2\par}\endgroup}%
  \def\urladdr##1##2{\begingroup
    \def~{\char`\~}%
    \@ifnotempty{##2}{\nobreak\noindent\urladdrname
      \@ifnotempty{##1}{, \ignorespaces##1\unskip}\/:\space
      \ttfamily##2\par}\endgroup}%
  \addresses
  \endgroup
  \global\let\addresses=\@empty
}
\def\@setabstracta{%
    \ifvoid\abstractbox
  \else
    \skip@25\p@ \advance\skip@-\lastskip
    \advance\skip@-\baselineskip \vskip\skip@
    \box\abstractbox
    \prevdepth\z@ 
    \vskip-10pt
  \fi
}
\renewenvironment{abstract}{%
  \ifx\maketitle\relax
    \ClassWarning{\@classname}{Abstract should precede
      \protect\maketitle\space in AMS document classes; reported}%
  \fi
  \global\setbox\abstractbox=\vtop \bgroup
    \normalfont\small
    \list{}{\labelwidth\z@
      \leftmargin0pc \rightmargin\leftmargin
      \listparindent\normalparindent \itemindent\z@
      \parsep\z@ \@plus\p@
      
    }%
    \item[\hskip\labelsep\bfseries\abstractname.]%
}{%
  \endlist\egroup
  \ifx\@setabstract\relax \@setabstracta \fi
}
\def\section{\@startsection{section}{1}%
  \z@{-1.2\linespacing\@plus-.5\linespacing}{.8\linespacing}%
  {\normalfont\bfseries\large}}
\def\subsection{\@startsection{subsection}{2}%
  \z@{-.8\linespacing\@plus-.3\linespacing}{.3\linespacing\@plus.2\linespacing}%
  {\normalfont\bfseries}}
\def\subsubsection{\@startsection{subsubsection}{3}%
  \z@{.7\linespacing\@plus.1\linespacing}{-1.5ex}%
  {\normalfont\itshape}}
\def\@secnumfont{\bfseries}
\makeatother
\fi 

\usepackage{amssymb}
\usepackage[all]{xy}
\usepackage{graphicx,color,float}
\usepackage[bookmarks]{hyperref}
\usepackage{pinlabel}
\usepackage{amsmath}
\usepackage{pb-diagram,pb-xy}
\usepackage{url}
\usepackage{epic,eepic}

\textwidth=\paperwidth \advance\textwidth by-2\oddsidemargin
\advance\oddsidemargin by-1in
\evensidemargin=\oddsidemargin
\calclayout

\def\tilde{\widetilde}

\def\lk{\operatorname{lk}}

\def\+{\oplus}

\theoremstyle{plain}
\newtheorem{theorem}{Theorem}[section]
\newtheorem{proposition}[theorem]{Proposition}
\newtheorem{corollary}[theorem]{Corollary}
\newtheorem{lemma}[theorem]{Lemma}
\theoremstyle{definition}

\newtheorem{example}[theorem]{Example}
\newtheorem{remark}[theorem]{Remark}

\makeatletter
\newcommand{\shortxra}[2][]{\ext@arrow 0359\rightarrowfill@{#1}{#2}}
\def\longrightarrowfill@{\arrowfill@\relbar\relbar\longrightarrow}
\newcommand{\longxra}[2][]{\ext@arrow 0359\longrightarrowfill@{#1}{#2}}

\makeatother

\newcommand{\tata}{
\includegraphics[bb= 0 50 163 160, scale=.09]{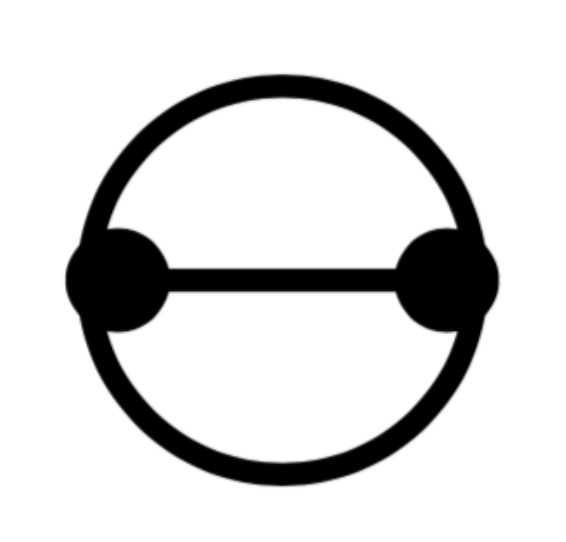}
}
\newcommand{\tetra}{
\includegraphics[bb = 0 60 223 205, scale=.09]{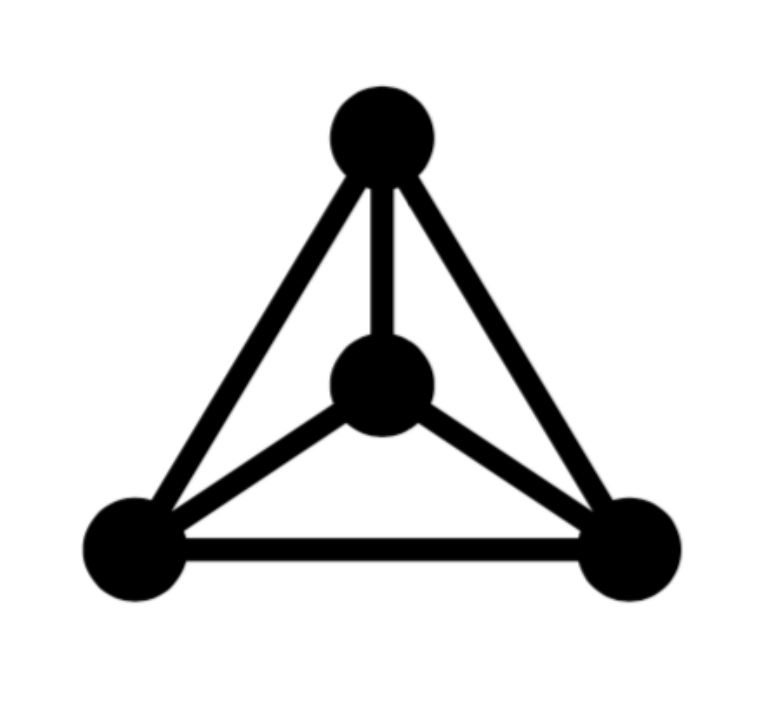}
}

\author{Kazuhiro Ichihara}
\address{Department of Mathematics\\
College of Humanities and Sciences, Nihon University\\
3-25-40 Sakurajosui, Setagaya-ku, Tokyo 156-8550, Japan.}
\email{ichihara@math.chs.nihon-u.ac.jp}

\author{Zhongtao Wu}
\address{Department of Mathematics\\
Lady Shaw Building\\
The Chinese University of Hong Kong\\
Shatin, Hong Kong}
\email{ztwu@math.cuhk.edu.hk}

\keywords{cosmetic surgery, Jones polynomial}
\subjclass{Primary 57M50, Secondary 57M25, 57M27, 57N10}

\begin{document}

\title{A note on Jones polynomial and cosmetic surgery}

\begin{abstract}
We show that two Dehn surgeries on a knot $K$ never yield manifolds that are homeomorphic as oriented manifolds if $V_K''(1)\neq 0$ or $V_K'''(1)\neq 0$.  As an application, we verify the cosmetic surgery conjecture for all knots with no more than $11$ crossings except for three $10$-crossing knots and five $11$-crossing knots.  We also compute the finite type invariant of order $3$ for two-bridge knots and Whitehead doubles, from which we prove several nonexistence results of purely cosmetic surgery.

\end{abstract}
\maketitle

\section{Introduction}

\emph{Dehn surgery} is an operation to modify a three-manifold by drilling and then regluing a solid torus.  Denote by $Y_r(K)$ the resulting three-manifold via Dehn surgery on a knot $K$ in $Y$ along a slope $r$. Two Dehn surgeries along $K$ with distinct slopes $r$ and $r'$ are called \emph{purely cosmetic} if $Y_{r}(K)\cong Y_{r'}(K)$ as oriented manifolds.  In Gordon's 1990 ICM talk \cite[Conjecture~6.1]{Gordon1990} and Kirby's Problem List \cite[Problem~1.81 A]{Kirby}, it is conjectured that two surgeries on inequivalent slopes are never purely cosmetic.  We shall refer to this as the \emph{cosmetic surgery conjecture}.

In the present paper we study purely cosmetic surgeries along knots in the three-sphere $S^3$.  We show that for most knots $K$ in $S^3$,  $S^3_r(K) \ncong S^3_{r'}(K)$ as oriented manifolds for distinct slopes $r$, $r'$.  More precisely, our main result gives a sufficient condition for a knot $K$ that admits no purely cosmetic surgery in terms of its Jones polynomial $V_K(t)$.

\begin{theorem}\label{main1}
If a knot $K$ has either $V_K''(1)\neq 0$ or $V_K'''(1)\neq 0$, then $S^3_r(K) \ncong S^3_{r'}(K)$ for any two distinct slopes $r$ and $r'$.

\end{theorem}

Here, $V_K''(1)$ and $V_K'''(1)$ denote the second and third order derivative of the Jones polynomial of $K$ evaluated at $t=1$, respectively. Note that in \cite[Proposition 5.1]{BoyerLines1990}, Boyer and Lines obtained a similar result for knots $K$ with $\Delta_K''(1)\neq 0$, where $\Delta_K(t)$ is the normalized Alexander polynomial.  We shall see that $V_K''(1)=-3\Delta_K''(1)$ (Lemma \ref{a2}).  Hence, our result can be viewed as an improvement of their result \cite[Proposition 5.1]{BoyerLines1990}.

Previously, other known classes of knots that are shown not to admit purely cosmetic surgeries include the genus $1$ knots \cite{Wang2006} and the knots with $\tau(K)\neq 0$ \cite{NiWu2015}, where $\tau$ is the concordance invariant defined by Ozsv\'ath-Szab\'o \cite{OzsvathSzabo2003} and Rasmussen \cite{Rasmussen2003} using Floer homology.  Theorem \ref{main1} along with the condition $\tau(K)\neq 0$ give an effective obstruction to the existence of purely cosmetic surgery.  For example, we used Knotinfo \cite{Knotinfo}, Knot Atlas \cite{Knotatlas} and Baldwin-Gillam's table in \cite{BaldwinGillam2012} to list all knots that have simultaneous vanishing $V_K''(1)$, $V_K'''(1)$ and $\tau$ invariant.  We get the following result:

\begin{corollary}
The cosmetic surgery conjecture is true for all knots with no more than $11$ crossings, except possibly  $$10_{33}, 10_{118}, 10_{146},$$
 $$11a_{91}, 11a_{138}, 11a_{285}, 11n_{86}, 11n_{157}.$$
\end{corollary}

\begin{remark}
In \cite{OzsvathSzabo2011}, Ozsv\'ath and Szab\'o gave the example of $K=9_{44}$, which is a genus two knot with $\tau(K)=0$ and $\Delta_K''(1)=0$.  Moreover, $S^3_1(K)$ and $S^3_{-1}(K)$ have the same Heegaard Floer homology, so no Heegaard Floer type invariant can distinguish these two surgeries.  This example shows that Theorem \ref{main1} and those criteria from Heegaard Floer theory are independent and complementary.
\end{remark}

The essential new ingredient in this paper is a surgery formula by Lescop, which involves a knot invariant $w_3$ that satisfies a crossing change formula \cite[Section 7]{Lescop2009}.  We will show that $w_3$ is actually the same as $\frac{1}{72}V_K'''(1)+\frac{1}{24}V_K''(1)$.  Meanwhile, we also observe that $w_3$ is a finite type invariant of order $3$.  This enables us to reformulate Theorem \ref{main1} in term of the finite type invariants of the knot (Theorem \ref{Cormain}).

As another application of Theorem \ref{main1}, we prove the nonexistence of purely cosmetic surgery on certain families of two-bridge knots and Whitehead doubles.  Along the way, an explicit closed formula for the canonically normalized finite type knot invariant of order 3
$$ v_3 ( K_{b_1 , c_1, \cdots , b_m, c_m} )=
\frac{1}{2}\left( \sum_{k=1}^m c_k (\sum_{i=1}^k b_i)^2 - \sum_{i=1}^m b_i (\sum_{k=i}^m c_k  )^2       \right)$$
is derived for two-bridge knots in Conway forms $K_{b_1 , c_1, \cdots , b_m, c_m}$ in Proposition \ref{v3formula}, which could be of independent interest.

The remaining part of this paper is organized as follows. In Section 2, we review background and properties of Jones polynomial, and prove crossing change formulae for derivatives of Jones polynomial.  In Section 3, we define an invariant $\lambda_2$ for rational homology spheres and then use Lescop's surgery formula to prove Theorem \ref{main1}.  In Section 4 and Section 5, we study in more detail cosmetic surgeries along two-bridge knots and Whitehead doubles. 

\vspace{5pt}\noindent{\bf Acknowledgements.}
The authors would like to thank Tomotada Ohtsuki and Ryo Nikkuni for stimulating discussions and drawing their attention to the reference \cite{Nikkuni2005}\cite{Ohtsuki}. The first named author is partially supported by JSPS KAKENHI Grant Number 26400100. The second named author is partially supported by grant from the Research Grants Council of Hong Kong Special Administrative Region, China (Project No. 14301215).


\section{Derivatives of Jones polynomial}
Suppose $(L_+$, $L_-$, $L_0)$ is a skein triple of links as depicted in Figure \ref{crossings}.

\begin{figure}[h!]
\includegraphics[scale=2]{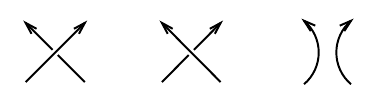}
\put(-188,-10){$ L_+$}
\put(-109,-10){$L_-$}
\put(-30, -10){$L_0$}

\caption{The link diagrams of $L_+$, $L_-$, $L_0$ are identical except at one crossing.} 
\label{crossings}
\end{figure}

\noindent
Recall that the \emph{Jones polynomial} satisfies the skein relation
\begin{equation} \label{Jones}
t^{-1} V_{L_+} (t) - t V_{L_-} (t) = ( t^{\frac{1}{2}} - t^{-\frac{1}{2}} ) V_{L_0} (t),
\end{equation}
and the \emph{Conway polynomial} satisfies the skein relation
\begin{equation}\label{Conway}
\nabla_{L_+}(z)-\nabla_{L_-}(z)=z\nabla_{L_0}(z).
\end{equation}
The \emph{normalized Alexander polynomial} $\Delta_L(t)$ is obtained by substituting $z=t^{1/2}-t^{-1/2}$ into the Conway polynomial.

For a knot $K$, denote $a_2(K)$ the $z^2$-term of the Conway polynomial $\nabla_K(z)$. It is not hard to see that $\Delta_K''(1)=2a_2(K)$.  If one differentiates Equations (\ref{Jones}) and (\ref{Conway}) twice and compares the corresponding terms, one can also show that $V_K''(1)=-6a_2(K)$.  See \cite{Murakami1986} for details. In summary, we have:

\begin{lemma}\label{a2}
For all knots $K\subset S^3$,
$$V_K''(1)=-6a_2(K)=-3\Delta_K''(1).$$

\end{lemma}

\bigskip

In \cite{Lescop2009}, Lescop defined an invariant $w_3$ for a knot $K$ in a homology sphere $Y$.  When $Y=S^3$, the knot invariant $w_3$ satisfies a crossing change formula
\begin{equation}\label{w3Def}
w_3(K_+) - w_3(K_-)
=\frac{a_2 (K') + a_2(K'')}{2} -
\frac{a_2 (K_+) + a_2(K_-) + \lk^2 (K' , K'')}{4},
\end{equation}
where $(K_+, K_-, K'\cup K'')$ is a skein triple consisting of two knots $K_{\pm}$ and a two-component link $K'\cup K''$  \cite[Proposition 7.2]{Lescop2009}.  Clearly, the values of $w_3(K)$ are uniquely determined by this crossing change formula once we fix $w_3$($=0$) for the unknot.  This gives an alternative characterization of the invariant $w_3$ for knots in $S^3$.  The next lemma relates it to the derivatives of Jones polynomial.

\begin{lemma}\label{w3}
For all knots $K\subset S^3$,
$$w_3(K)=\frac{1}{72}V_K'''(1)+\frac{1}{24}V_K''(1).$$

\end{lemma}

\begin{proof}
The main argument essentially follows from Nikkuni \cite[Proposition 4.2]{Nikkuni2005}.  We prove the lemma by showing that $\frac{1}{72}V_K'''(1)+\frac{1}{24}V_K''(1)$ satisfies an identical crossing change formula as Equation (\ref{w3Def}). To this end, we differentiate the skein formula for the Jones polynomial (\ref{Jones}) three times and evaluate at $t=1$.  Abbreviating the Jones polynomial of the skein triple $L_+=K_+$, $L_-=K_-$ and $L_0=K'\cup K''$ by $V_+ (t)$, $V_- (t)$ and $V_0 (t)$, respectively, we obtain

\begin{eqnarray*}
\left( t^{-1} V_+ (t) \right)'''|_{t=1}
&=&
-6 V_+ (1) + 6 V'_+ (1) -3 V''_+ (1) + V'''_+(1) \\
\left( t V_- (t) \right)''' |_{t=1}
&=&
3 V''_- (1) + V'''_- (1) \\
\left( (t^{1/2} - t^{-1/2} )  V_0 (t) \right)''' |_{t=1}
&=&
\frac{9}{4} V_0 (1) -3 V'_0 (1) + 3 V''_0 (1)
\end{eqnarray*}
The terms on the right hand side can be expressed as

\begin{itemize}
\item[(a)] $V_+ (1) = V_- (1) = 1$
\item[(b)] $V'_+ (1) = V'_- (1) = 0$
\item[(c)] $V''_+ (1) = -6 a_2 (K_+) $, $V''_- (1) = -6 a_2 (K_-) $
\item[(d)] $V_0 (1) = -2$
\item[(e)] $V'_0 (1) = -3 \lk (K', K'')$
\item[(f)] $V''_0 (1) = -\frac{1}{2} +3 \lk (K', K'') + 12 ( a_2 (K') + a_2 (K'')) -6 \lk^2(K', K'') $
\end{itemize}
Here, (a) and (d) are well-known; (b),(c),(e) and (f) are proved by Murakami \cite{Murakami1986}.\footnote{Murakami uses a different skein relation for the Jones polynomial, thus (e) and (f) differ by certain signs from the formula in \cite{Murakami1986}.}  After doing substitution and simplification, we have
$$
V_{+}'''(1) - V_{-}'''(1)
=
- 18\left(a_{2}(K_{+}) + a_{2}(K_{-})\right)
- 18\lk^2(K', K'')
- 18\lk(K', K'')
+ 36\left(a_{2}(K') + a_{2}(K'')\right)
$$
Meanwhile, it follows from (\ref{Conway}) and Hoste \cite[Theorem 1]{Hoste1985} that
\begin{equation}\label{a2crossingchange}
\lk(K' , K'') = a_{2}(K_{+}) - a_{2}(K_{-}).
\end{equation}
This enables us to further simplify
$$(\frac{1}{72} V_+'''(1) + \frac{1}{24} V_+''(1))-(\frac{1}{72} V_-'''(1) + \frac{1}{24} V_-''(1))$$
and reduce it to the same expression as the right hand side of (\ref{w3Def}).
As $\frac{1}{72} V_K'''(1) + \frac{1}{24} V_K''(1)$ also equals $0$ when $K$ is the unknot,  $\frac{1}{72}V_K'''(1)+\frac{1}{24}V_K''(1)$ must equal $w_3(K)$ for all $K\subset S^3$.

\end{proof}

\bigskip
We conclude the section by remarking that both Lemma \ref{a2} and Lemma \ref{w3} can be seen in a simpler way from a more natural perspective.  A knot invariant $v$ is called a \emph{finite type invariant} of order $n$ if it can be extended to an invariant of singular knots via a skein relation
$$v(\tilde{K})=v(K_+)-v(K_-)$$
where $\tilde{K}$ is the knot with a transverse double point (See Figure \ref{crossings2}), while $v$ vanishes for all singular knots with $(n+1)$ singularities.

\begin{figure}[h!]
\includegraphics[scale=2]{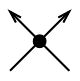}
\caption{the singular knot $\tilde{K}$ with a transverse double point        }
\label{crossings2}
\end{figure}

It follows readily from the definition that the set of finite type invariant of order $0$ consists of all constant functions.  One can also show that $a_2(K)$ and $V_K''(1)$ are finite type invariants of order $2$, while $w_3(K)$ and $V_K'''(1)$ are finite type invariants of order $3$.  As the dimension of the set of all finite type invariants of order $\leq 2$ and $\leq 3$ are two and three, respectively (see, e.g., \cite{Bar-Natan1995}), there has to be a linear dependence among the above knot invariants, from which one can easily deduce Lemma \ref{a2} and Lemma \ref{w3}.  In fact, if we denote $v_2$ and $v_3$ the finite type invariants of order $2$ and $3$ respectively normalized by the conditions that $v_2(m(K))=v_2(K)$ and $v_3(m(K))=-v_3(K)$ for any knot $K$ and its mirror image $m(K)$ and that $v_2(3_1)=v_3(3_1)=1$ for the right hand trefoil $3_1$, then it is not difficult to see that
\begin{equation}\label{v2}
v_2(K)=a_2(K)
\end{equation}
and
\begin{equation}\label{v3}
v_3(K)=-2w_3(K).
\end{equation}

\section{Lescop invariant and cosmetic surgery}

The goal of this section is to prove Theorem \ref{main1}.  Recall the following results about purely cosmetic surgery from \cite[Proposition 5.1]{BoyerLines1990} and \cite[Theorem 1.2]{NiWu2015}.

\begin{theorem}\label{MainLemma}
Suppose $K$ is a nontrivial knot in $S^3$, $r,r'\in\mathbb Q\cup\{\infty\}$ are two distinct slopes such that $S^3_{r}(K)\cong S^3_{r'}(K)$
as oriented manifolds. Then the following assertions are true:
\begin{itemize}
\item[(a)] $\Delta_K''(1)=0$.
\item[(b)] $r=-r'$.
\item[(c)] If $r=p/q$, where $p,q$ are coprime integers, then $q^2 \equiv -1\; (\bmod\; p)$.
\item[(d)] $\tau(K)=0$, where $\tau$ is the concordance invariant defined by Ozsv\'ath-Szab\'o \cite{OzsvathSzabo2003} and Rasmussen \cite{Rasmussen2003}.
\end{itemize}

\end{theorem}

\bigskip
Our new input for the cosmetic surgery problem is Lescop's $\lambda_2$ invariant which, roughly speaking, is the degree $2$ part of the Kontsevich-Kuperberg-Thurston invariant of rational homology spheres \cite{Lescop2009}.  Like the famous Le-Murakami-Ohtsuki invariant, the Kontsevich-Kuperberg-Thurston invariant is universal among finite type invariants for homology spheres \cite{Kontsevich1994}\cite{KuperbergThurston}\cite{LMO1998}.  See also Ohtsuki \cite{Ohtsuki} for the connection to perturbative and quantum invariants of three-manifolds.

We briefly review the construction.  A \emph{Jacobi diagram} is a graph without simple loop whose vertices all have valency $3$. The degree of a Jacobi diagram is defined to be half of the total number of vertices of the diagram.  If we denote by $\mathcal{A}_n$ the vector space generated by degree $n$ Jacobi diagrams subject to certain equivalent relations AS and IHX, then the degree $n$ part $Z_n$ of the Kontsevich-Kuperberg-Thurston invariant takes its value in $\mathcal{A}_n$.

\begin{example}
Simple argument in combinatorics implies that
\begin{itemize}
\item  $\mathcal{A}_1$ is an $1$-dimensional vector space generated by the Jacobi diagram $\tata$
\item $\mathcal{A}_2$ is a $2$-dimensional vector space generated by the Jacobi diagrams $\tata\; \tata$ and $\tetra$

\end{itemize}

\end{example}

Many interesting real invariants of rational homology spheres can be recovered from the Kontsevich-Kuperberg-Thurston invariant $Z$ by composing a linear form on the space of Jacobi diagrams.  In the simplest case, the Casson-Walker invariant $\lambda_1$  is $W_1\circ Z_1$, where $W_1(\tata)=2$.  We shall concentrate on the case of the degree $2$ invariant $\lambda_2= W_2 \circ Z_2$, where $W_2(\;\tetra)=1$ and $W_2(\tata \;\tata)=0$.  The following surgery formula for $\lambda_2$ is proved by Lescop and will play a central role in the proof of our main result.

\begin{theorem}\cite[Theorem 7.1]{Lescop2009}
\label{LescopSurgery} The invariant $\lambda_2$ satisfies the surgery formula
$$\lambda_2(Y_{p/q}(K))-\lambda_2(Y)=\lambda_2''(K)(\frac{q}{p})^2+w_3(K)\frac{q}{p}
+a_2(K)c(q/p)+\lambda_2(L(p,q))$$
for all knots $K\subset Y$.
\end{theorem}

Here, $a_2(K)$ is the $z^2$-coefficient of $\nabla_K(z)$, and $L(p,q)$ is the lens space obtained by $p/q$ surgery on the unknot.\footnote{We use a different sign convention of lens spaces from Lescop's original paper.} Then $w_3(K)$ is a knot invariant, which was shown earlier in Lemma \ref{w3} to be equal to $\frac{1}{72}V_K'''(1)+\frac{1}{24}V_K''(1)$ for $K\subset S^3$.  The terms $\lambda_2''(K)$ and $c(q/p)$ are both explicitly defined in \cite{Lescop2009}, but they will not be needed for our purpose.  For the moment, we make the following simple observation.

\begin{proposition}\label{w3Cosmetic}
Suppose $K$ is a knot in $S^3$ with $a_2(K)=0$, and $p, q$ are nonzero integers satisfying $q^2 \equiv -1 \;(\bmod\; p)$.  Then $\lambda_2(S^3_{p/q}(K)) = \lambda_2(S^3_{-p/q}(K))$ if and only if $w_3(K)=0$.

\end{proposition}

\begin{proof}
We apply the surgery formula in Theorem \ref{LescopSurgery}. Note that the first and third terms of the right hand side are clearly equal for $p/q$ and $-p/q$ surgery.  Next, recall the well-known theorem that two lens spaces $L(p, q_1)$ and $L(p, q_2)$ are equivalent up to orientation-preserving homeomorphisms if and only if $q_1 \equiv q_2^{\pm1} \;(\bmod\; p)$.  In particular, this implies the lens spaces $L(p, q)\cong L(p, -q)$ as oriented manifolds if $q^2 \equiv -1 \;(\bmod \; p)$, so their $\lambda_2$ invariants are obviously the same.  Consequently, $$\lambda_2(S^3_{p/q}(K))- \lambda_2(S^3_{-p/q}(K))=w_3(K)\frac{2q}{p}, $$ and the statement follows readily.

\end{proof}

\begin{proof}[Proof of Theorem \ref{main1}]
In light of Theorem \ref{MainLemma},  we only need to consider the case when $\Delta_K''(1)=0$ and $q^2\equiv -1 \;(\bmod \;p)$, for otherwise, the pair of manifolds $S^3_{p/q}(K)$ and $S^3_{-p/q}(K)$ will be non-homeomorphic as oriented manifolds.  Thus $V_K''(1)=-3\Delta_K''(1)=0$. If we now assume $V_K'''(1) \neq 0$, then Lemma \ref{w3} implies that $w_3(K) \neq 0$.  We can then apply Proposition \ref{w3Cosmetic} and conclude that $\lambda_2(S^3_{p/q}(K)) \neq \lambda_2(S^3_{-p/q}(K))$.  Consequently, $S^3_{p/q}(K) \ncong S^3_{-p/q}(K)$. 

\end{proof}

Given (\ref{v2}) and (\ref{v3}), Theorem \ref{main1} can be stated in the following equivalent way, which is particularly useful in the case where it is easier to calculate the finite type invariant $v_3$ (or equivalently $w_3$) than the Jones polynomial.

\begin{theorem}\label{Cormain}
If a knot $K$ has the finite type invariant $v_2(K)\neq0$ or $v_3(K)\neq 0$, then $S^3_r(K) \ncong S^3_{r'}(K)$ for any two distinct slopes $r$ and $r'$.
\end{theorem}

\section{Examples of two-bridge knots}

In this section, we derive an explicit formula for $v_3$ and use it to study the cosmetic surgery problem for two-bridge knots.  Following the presentation of \cite[Section 2.1]{KawauchiBook}, we sketch the basic properties and notations for two-bridge knots.

Every two-bridge knot can be represented by a rational number $-1<\frac{\alpha}{\beta}<1$ for some odd integer $\alpha$ and even integer $\beta$.  If we write this number as a continued fraction with even entries and of even length
$$\frac{\alpha}{\beta}=[2b_1, 2c_1, \cdots, 2b_m, 2c_m]=2b_1 + \cfrac{1}{2c_1
          + \cfrac{1}{\cdots
          + \cfrac{1}{2b_m + \cfrac{1}{2c_m} } } }$$
for some nonzero integers $b_i$'s and $c_i$'s,\footnote{Such a representation always exists by elementary number theory.} then we obtain the \textit{Conway form} $C(2b_1, 2c_1, \cdots , 2b_m , 2c_m)$ of the two-bridge knot, which is a special knot diagram as depicted in Figure~\ref{ConwayForm}.  We will write $K_{b_1, c_1, \cdots, b_m, c_m}$ for the knot of Conway form $C(2b_1, 2c_1, \cdots , 2b_m , 2c_m)$. The genus of $K_{b_1, c_1, \cdots , b_m , c_m}$ is $m$; and conversely, every two-bridge knot of genus $m$ has such a representation.

\begin{figure}[htb]\centering
\includegraphics{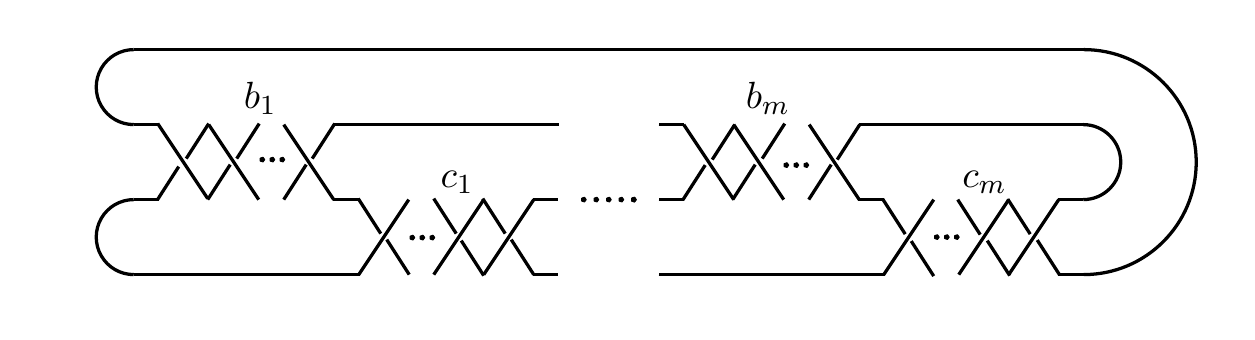}
\caption{This is the knot diagram of the Conway form $C(2b_1, 2c_1, \cdots , 2b_m , 2c_m)$ of a two bridge knot.   In the figure, there are $| b_i |$ positive (resp. negative) full-twists if $b_i >0$
(resp $b_i <0$),
and there are $| c_j |$ negative (resp. positive) full-twists if $c_j >0$ (resp. $c_j <0$) for $1 \le i,j \le n$.
  }
\label{ConwayForm}
\end{figure}

\bigskip\noindent
Burde obtained the following formula for $a_2(K_{b_1, c_1, \cdots, b_m, c_m})$, the $z^2$-coefficient of the Conway polynomial of $K_{b_1, c_1, \cdots , b_m , c_m}$.

\begin{proposition}\cite[Proposition 5.1]{Burde1990}\label{Burde}
For the two-bridge knot $K_{b_1, c_1, \cdots, b_m, c_m}$, the $z^2$-coefficient of the Conway polynomial is given by $$a_2(K_{b_1, c_1, \cdots, b_m, c_m})=-\sum_{i=1}^m\sum_{k=i}^m b_ic_k= -\sum_{k=1}^{m}\sum_{i=1}^k c_k b_i.$$

\end{proposition}

The above formula can be proved by recursively applying Equation (\ref{a2crossingchange}). The similar idea can be used to find an analogous formula for $w_3$, which is the main task of the next few lemmas.

\begin{lemma}\label{w3RecursiveLemma}
The invariant $w_3$ satisfies the recursive formula
\begin{eqnarray*}
&& w_3 ( K_{b_1 , c_1, \cdots , b_m, x} ) - w_3 ( K_{b_1 , c_1, \cdots , b_{m}, x-1} ) \\
&=& -\frac{1}{4} \left( a_2  (K_{b_1 , c_1, \cdots , b_m, x}) + a_2 ( K_{b_1 , c_1, \cdots , b_{m}, x-1} ) + (\sum_{i=1}^{m}b_i)^2 \right)
\end{eqnarray*}

\end{lemma}

\begin{proof}
This follows from a direct application of the crossing change formula (\ref{w3Def}) at the rightmost crossing in Figure \ref{ConwayForm}, and the observation that both $K'$ and $K''$ are the unknot with $\lk(K', K'')=-\sum_{i=1}^{m}b_i$.

\end{proof}

\begin{lemma}\label{w3Lemma2}
The invariant $w_3$ satisfies the recursive formula
\begin{align*} & w_3 ( K_{b_1 , c_1, \cdots , b_m, c_m} ) - w_3 ( K_{b_1 , c_1, \cdots , b_{m-1}, c_{m-1}} )  \\
&=- \frac{1}{4} \left( 2c_m \cdot a_2 ( K_{b_1 , c_1, \cdots , b_{m-1 }, c_{m-1} } ) - c_m^2 \sum_{i=1}^m b_i+ c_m (\sum_{i=1}^{m}b_i)^2 \right)
\end{align*}

\end{lemma}

\begin{proof}
We first prove the lemma for $c_m>0$.  We repeatedly apply Lemma \ref{w3RecursiveLemma} until $x$ is reduced to $0$.  Note that the knot $K_{b_1, c_1, \cdots, b_m, 0}$ can be isotoped to $K_{b_1, c_1, \cdots, b_{m-1}, c_{m-1}}$ by untwisting the far-right $b_m$ full twists. Therefore,

\begin{align*} & w_3 ( K_{b_1 , c_1, \cdots , b_m, c_m} ) - w_3 ( K_{b_1 , c_1, \cdots , b_{m-1}, c_{m-1}} )  \\
&= -\frac{1}{4} \left( a_2 ( K_{b_1 , c_1,  \cdots , b_m, c_m} )
+ 2 \sum_{x=1}^{c_{m} - 1 } a_2 ( K_{b_1 , c_1,  \cdots , b_m, x} )
+ a_2 ( K_{b_1 , c_1, \cdots , b_{m-1 }, c_{m-1} } ) + c_m (\sum_{i=1}^{m}b_i)^2 \right)
\end{align*}
Now, the lemma follows from substituting $$a_2(K_{b_1, c_1, \cdots, b_m, x})=a_2(K_{b_1, c_1, \cdots, b_{m-1}, c_{m-1}})-x\sum_{i=1}^m b_i,$$ which is an immediate corollary of Proposition \ref{Burde}.

The case when $c_m<0$ is proved analogously. 
\end{proof}
Finally, applying Lemma \ref{w3Lemma2} and induction on $m$, we obtain an explicit formula for $w_3$, and consequently also for $v_3$.

\begin{proposition} \label{v3formula}

$$ v_3 ( K_{b_1 , c_1, \cdots , b_m, c_m} )=-2w_3 ( K_{b_1 , c_1, \cdots , b_m, c_m} )=
\frac{1}{2}\left(\sum_{k=1}^m c_k (\sum_{i=1}^k b_i)^2 - \sum_{i=1}^m b_i (\sum_{k=i}^m c_k  )^2          \right)$$

\end{proposition}

\begin{proof}
We use induction on $m$.  For the base case $m=1$, Lemma \ref{w3Lemma2} readily implies that
$$w_3(K_{b_1, c_1})=-\frac{1}{4}(c_1b_1^2-c_1^2b_1),$$
so $K_{b_1, c_1}$ satisfies the formula.

Next we prove that if the formula holds for $K_{b_1 , c_1, \cdots , b_{m-1}, c_{m-1}}$, then it also holds for $K_{b_1 , c_1, \cdots , b_m, c_m}$.  It suffices to show that
\begin{align*} & -\frac{1}{4}\left(\sum_{k=1}^m c_k (\sum_{i=1}^k b_i)^2 - \sum_{i=1}^m b_i (\sum_{k=i}^m c_k  )^2          \right)
+\frac{1}{4}\left(\sum_{k=1}^{m-1} c_k (\sum_{i=1}^k b_i)^2 - \sum_{i=1}^{m-1} b_i (\sum_{k=i}^{m-1} c_k  )^2          \right)
  \\
&= - \frac{1}{4} \left( 2c_m \cdot a_2 ( K_{b_1 , c_1, \cdots , b_{m-1 }, c_{m-1} } ) - c_m^2 \sum_{i=1}^m b_i+ c_m (\sum_{i=1}^{m}b_i)^2 \right)
\end{align*}
where $$a_2(K_{b_1, c_1, \cdots, b_{m-1}, c_{m-1}})=-\sum_{i=1}^{m-1}\sum_{k=i}^{m-1} b_ic_k.$$
The above identity can be verified from tedious yet elementary algebra.  We omit the computation here.

\end{proof}


For the rest of the section, we apply Theorem \ref{Cormain} and Proposition \ref{v3formula} to study the cosmetic surgery problems for the two-bridge knots of genus $2$ and $3$, which correspond to the Conway form $K_{b_1, c_1, b_2, c_2}$ and $K_{b_1, c_1, b_2, c_2, b_3, c_3}$, respectively.  Note that the cosmetic surgery conjecture for genus one knot is already settled by Wang \cite{Wang2006}.

\begin{corollary}
If a genus $2$ two-bridge knot $K_{b_1, c_1, b_2, c_2}$ is not of the form $K_{x,y,-x-y,x}$ for some integers $x,y$, then it does not admit purely cosmetic surgeries.
\end{corollary}

\begin{proof}
Suppose there are purely cosmetic surgeries for the knot $K_{b_1, c_1, b_2, c_2}$.  Theorem \ref{Cormain} implies that
\begin{equation}\label{v2genus2}
 a_2 (K_{b_1, c_1, b_2, c_2})= - ( b_1 c_1 + b_1 c_2 + b_2 c_2)=0,
\end{equation}
and
\begin{equation}\label{v3genus2}
v_3 (K_{b_1, c_1, b_2, c_2})=
\frac{1}{2} \left( c_1 b_1^2 + c_2 ( b_1 + b_2 )^2 - b_1 ( c_1 + c_2 )^2 - b_2 c_2^2 \right)=0,
\end{equation}
where the formula for $a_2$ and $v_3$ follows from Proposition \ref{Burde} and Proposition \ref{v3formula}, respectively.  From Equation (\ref{v2genus2}), we see $c_2(b_1+b_2)=-b_1c_1$ and $b_1(c_1+c_2)=-b_2c_2$, which was then substituted into the second and the third terms of Equation (\ref{v3genus2}), and gives
$$v_3 (K_{b_1, c_1, b_2, c_2})=\frac{1}{2} \left( c_1 b_1^2 -b_1c_1( b_1 + b_2 ) +b_2c_2( c_1 + c_2 ) - b_2 c_2^2 \right)=\frac{1}{2}b_2c_1(c_2-b_1)=0.
$$
Hence, $b_1=c_2$.  Plugging this identity back to Equation (\ref{v2genus2}), we see $b_1+b_2+c_1=0$.  As a result, the two-bridge knot $K_{b_1, c_1, b_2, c_2}$ can be written as $K_{x,y,-x-y,x}$ for some integers $x$ and $y$.

\end{proof}

We can perform a similar computation for a genus $3$ two-bridge knot $K_{b_1, c_1, b_2, c_2, b_3, c_3}$. By Proposition \ref{v3formula},
$$v_3(K_{b_1, c_1, b_2, c_2, b_3, c_3})
=\frac{1}{2}(c_1b_1^2+c_2(b_1+b_2)^2+c_3(b_1+b_2+b_3)^2-b_1(c_1+c_2+c_3)^2
-b_2(c_2+c_3)^2-b_3c_3^2).$$
In particular, we see
$$v_3(K_{x,1,-x,x,1,-x})=-x\neq 0.$$
Consequently, Theorem \ref{Cormain} implies

\begin{corollary}

The family of two-bridge knots $K_{x,1,-x,x,1,-x}$ does not admit purely cosmetic surgeries.

\end{corollary}

\begin{remark}
As explained in \cite{IchiharaSaito2016}, both $\Delta''_K(1)$ and $\tau(K)$ are $0$ for the knot $K_{x,1,-x,x,1,-x}$. Hence, purely cosmetic surgery could not be ruled out by previously known results from Theorem \ref{MainLemma}.

\end{remark}

\bigskip

\section{Examples of Whitehead doubles}
We are devoted to $D_+(K,n)$ in this section, where $D_+(K,n)$ denotes the satellite of $K$ for which the pattern is a positive-clasped twist knot with $n$ twists.  The knot $D_+(K,n)$ is called the \emph{positive $n$-twisted Whitehead double} of a knot $K$.  See Figure \ref{Whitehead} for an illustration.

\begin{figure}[H]
\centering
\includegraphics{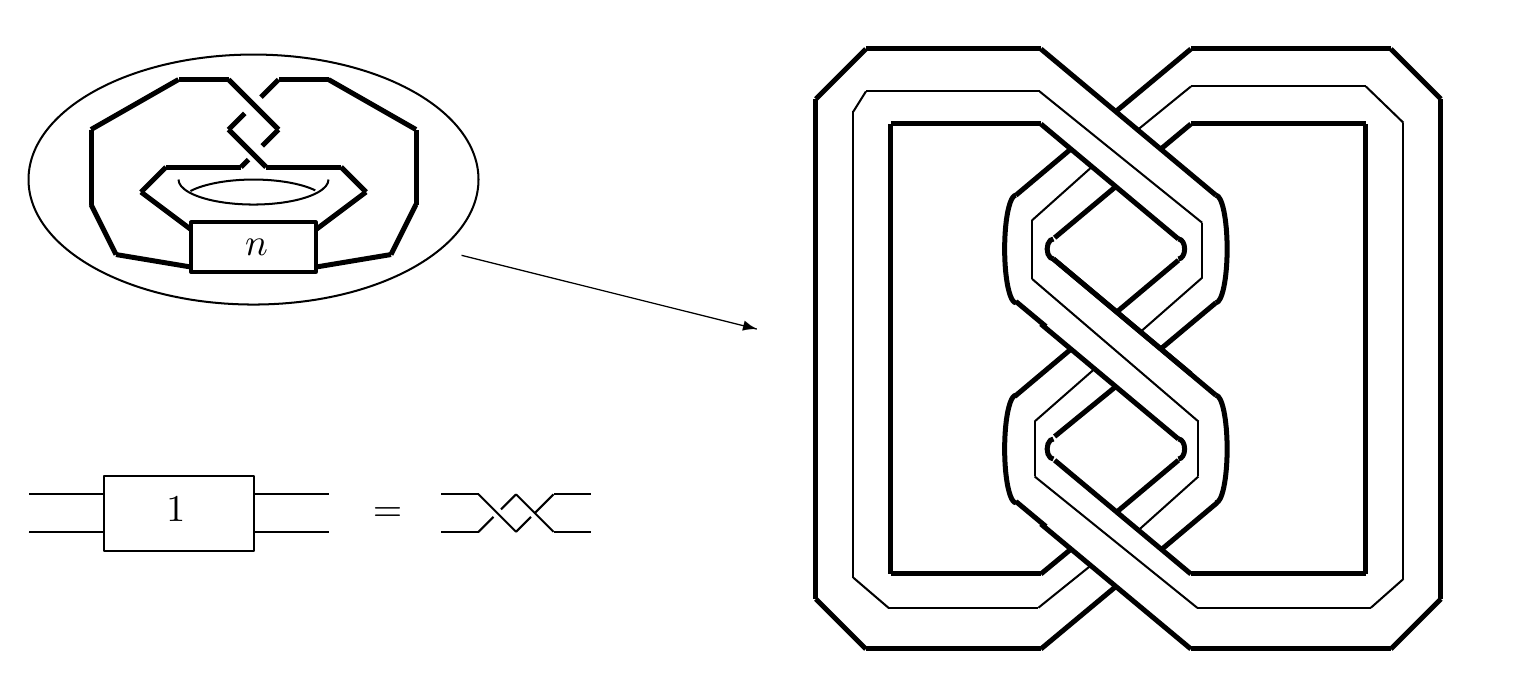}
\caption{the positive $n$-twisted Whitehead double $D_+(K,n)$}
\label{Whitehead}
\end{figure}

We perform the following calculation, which gives a mild generalization of \cite[Proposition 7.3]{Stoimenow2003}.

\begin{proposition} \label{v3Whitehead}
Suppose $D_+(K,n)$ is the positive $n$-twisted Whitehead double of a knot $K$.  Then
$$v_3(D_+(K, n))= -2a_2(K)+\frac{-n+n^2}{2} .$$
In particular, for the untwisted Whitehead doubles $D_+(K,0)$, $$v_3(D_+(K, 0))=-2a_2(K).$$

\end{proposition}

\begin{proof}
We apply the crossing change formula (\ref{w3Def}) at either one of the crossings of the clasps.  Note that $K_+=D_+(K,n)$, $K_-$ is the unknot, and $K'=K''=K$.
The classical formula for the Alexander polynomial of a satellite knot implies that $\Delta_{D_+(K, n)}(t)=-nt+(2n+1)-nt^{-1}$, from which we compute
$$a_2(D_+(K, n))=\frac{1}{2}\Delta_{D_+(K, n)}''(1)=-n.$$
Also observe that $\lk(K', K'')=-n$.  Therefore,
$$w_3(D_+(K, n))=a_2(K)-\frac{ -n+n^2}{4},$$
and so $$v_3(D_+(K, n))=-2a_2(K)+\frac{-n+n^2}{2}.$$

\end{proof}

Since the invariant $a_2(D_+(K, n))=-n$, the Whitehead double $D_+(K, n)$ does not admit purely cosmetic surgeries if $n\neq 0$.  When $n=0$,  Proposition \ref{v3Whitehead} gives $v_3(D_+(K, n))=-2a_2(K)$.  Hence, Theorem \ref{Cormain} immediately implies the following corollary.

\begin{corollary}
There is no purely cosmetic surgery for the positive $n$-twisted Whitehead double $D_+(K, n)$ for $n\neq 0$.  Moreover, if $a_2(K)\neq 0$, then there is no purely cosmetic surgery for the untwisted Whitehead double $D_+(K, 0)$ .
\end{corollary}

\begin{remark}
Note that $a_2(D_+(K,0))=0$.  Hedden shows $\tau(D_+(K,0))=0$ when $\tau(K) \leq 0$ \cite[Theorem 1.5]{Hedden2007}.
\end{remark}

\end{document}